\documentclass[reqno]{amsart}
\usepackage{amsmath,amssymb,amsthm}
\usepackage{bbm}
\usepackage{longtable}
\usepackage[left=2cm, right=2cm, top=2cm, bottom=2cm]{geometry}
\theoremstyle{definition}
\newtheorem{defn}{Definition}

\newtheorem{theorem}{Theorem}
\newtheorem{lemma}[defn]{Lemma}

\newtheorem{prop}[defn]{Proposition}
\newtheorem{rmk}[defn]{Remark}
\newtheorem*{qn*}{Question}

\numberwithin{defn}{section}
\newcommand{\gap}{\, \vspace{1em}}

\newcommand{\F}{\mathbb{F}}
\newcommand{\M}{\mathcal{M}}
\newcommand{\I}{\mathcal{I}}

\let\oldproofname=\proofname
\renewcommand{\proofname}{\rm\bf{\oldproofname}}
\renewcommand{\leq}{\leqslant}
\renewcommand{\geq}{\geqslant}
\renewcommand{\subset}{\subseteq}

\begin{document}

\pagenumbering{arabic}


\title{Highly Composite Polynomials and the maximum order of the divisor function in $\F_q[t]$}
\author[A. Afshar]{Ardavan Afshar}
\address{Department of Mathematics\\University College London\\
25 Gordon Street, London, England}
\email{ardavan.afshar.15@ucl.ac.uk}

\maketitle

\begin{abstract}
We investigate the analogues, in $\F_q[t]$, of highly composite numbers and the maximum order of the divisor function, as studied by Ramanujan. In particular, we determine a family of highly composite polynomials which is not too sparse, and we use it to compute the logarithm of the maximum of the divisor function at every degree up to an error of a constant, which is significantly smaller than in the case of the integers, even assuming the Riemann Hypothesis.
\end{abstract}

\gap

\section{Introduction} 

\gap
 
In \cite{Ramanujan}, Ramanujan investigated the divisor function $d(n)$, the number of divisors of $n$. Being interested in the maximum order of $d(n)$, he defined highly composite integers $n$ to be those for which $d(n) > d(n')$ for all $n > n'$, so that $D(N) := \max\{d(n) \ | \ n \leq N\}$ is given by $d(n')$ for the largest highly composite $n'\leq N$. He was able to compute $\log D(N)$ up to an error of at most $O(e^{-c\sqrt{\log \log N}}\log N)$ unconditionally and $O\left(\frac{\sqrt{\log N}}{(\log \log N)^3}\right)$ assuming the Riemann Hypothesis. Ramanujan studied carefully the prime factorisation of the highly composite integers, and his results were improved by Alaoglu and Erd\H{o}s in \cite{Alaoglu-Erdos}, who determined the exponent of each prime in the factorisation of a highly composite number up to an error of at most 1. \\

We consider the question of maximising the divisor function in the function field setting. Let $\mathbb{F}_q$ be a finite field, $\M = \{f \in \mathbb{F}_q[t] \text{ monic}\} $, $\M_n = \{f \in \M  :   \deg f = n \} $, $\I = \{f \in \M \text{ irreducible}\}$, $\I_n = \{f \in \I  : \deg f = n\}$, and $\pi(n) = |\I(n)| = \frac{1}{n} \sum_{d | n} \mu(d) q^{n/d}$ where $\mu(d)$ is the M\"{o}bius function. For $f \in \M$, let $\tau(f)$ be the number of monic divisors of $f$, and observe that a generic polynomial $f$ in $\M$ is of the form 
\begin{equation} \label{eqn:polynomial}
f= \prod_{p \in \I} p^{a_p} \quad \text{ with } \quad \deg f = \sum_{p \in \I} a_p \deg p \quad \text{ and } \quad \tau(f) =  \prod_{p \in \I} (1+a_p)
\end{equation}
where only finitely many $a_p$ are non-zero.  We wish to understand the polynomials which maximise the function $\tau$ up to a given degree, defined thus:

\begin{defn}
We call $f \in \M$ a \emph{highly composite polynomial} of degree $n$ if $\tau(f) = \max\{\tau(g) \ | \ g \in \bigcup_{m \leq n}\M_m\}$.
\end{defn}
\begin{rmk}
Highly composite polynomials of a given degree are not necessarily unique. For example, all linear polynomials in $\M_1$ are highly composite polynomials of degree 1. 
\end{rmk}
\begin{rmk} \label{newHCP}
There is (at least) one new highly composite polynomial at each degree. Indeed, let $f$ be a highly composite polynomial of degree $n$ and suppose otherwise, so that $\deg f = m < n$. Then pick some $g \in \M_{n-m}$, so that $fg \in \M_n$ but $\tau(fg) = \#\{d\in\M:d|fg\} \geq \#(\{d\in\M:d|f\} \cup \{fg\}) > \tau(f)$, which is a contradiction.
\end{rmk}
\begin{rmk} \label{increasing}
If $f= \prod_{p \in \I} p^{a_p}$ is a highly composite polynomial, then $\deg p_i < \deg p_j$ implies $a_{p_i} \geq a_{p_j}$. Indeed, suppose otherwise and set $g = fp_i^{a_{p_j}-a_{p_i}}p_j^{a_{p_i}-a_{p_j}}$, so that $\tau(g) = \tau(f)$ but $\deg g = \deg f - (a_{p_j}-a_{p_i})(p_j - p_i) < \deg f$, which contradicts Remark \ref{newHCP}.
\end{rmk}
\begin{rmk}
See Appendix \ref{appendix} for an illustrative table of highly composite polynomials in $\F_2[t]$.
\end{rmk}

\gap

We first investigate a family of highly composite polynomials $\{h(x)\}_{x>0}$ (following Ramanujan in \cite{Ramanujan}), which we define as follows:

\begin{defn}
Let $x>0$. We say that $h=h(x) \in \M$ is an $x$-\emph{superior highly composite} polynomial, or just $x$-SHC, if for all $f \in \M$ we have
\begin{equation} \label{eqn:SHC}
\frac{\tau(h)}{q^{\deg h/x}} \begin{cases} \geq \frac{\tau(f)}{q^{\deg f/x}} &\text{ if } \deg h \geq \deg f \\ > \frac{\tau(f)}{q^{\deg f/x}} &\text{ if } \deg h < \deg f \end{cases}
\end{equation}
and we say that $h$ is an $x$-\emph{semi-superior highly composite} polynomial, or $x$-SSHC, if for all $f \in \M$ we have
\begin{equation} \label{eqn:SSHC}
\frac{\tau(h)}{q^{\deg h/x}} \geq \frac{\tau(f)}{q^{\deg f/x}}.
\end{equation}
A polynomial which is $x$-SHC or $x$-SSHC for some $x>0$ is called \emph{superior highly composite} or \emph{semi-superior highly composite} respectively.
\end{defn}
\begin{rmk}
Clearly, if $h \in \M$ is $x$-SHC, then it is $x$-SSHC. Moreover, any polynomial $h$ which is $x$-SSHC is highly composite, since if $f \in \M$ with $\deg f \leq \deg h$, then by equation \eqref{eqn:SSHC} we have that
$$ \tau(f) \leq \frac{\tau(h)}{q^{(\deg h - \deg f)/x}} \leq \tau(h). $$
\end{rmk}

\gap

After defining a particular set for the parameter $x$ of an $x$-SSHC:

\begin{defn}
Let $$S = S_q := \left\{\frac{s \log q}{\log(1 + 1/r)} \ : \ s, r \geq 1 \right\}.$$
\end{defn}

\gap

we are able to determine the structure of the superior highly composite polynomials

\begin{theorem} \label{mainSHC}
 Let $x > 0$.
\begin{enumerate}
\item \label{SHCunique} There is one, and only one, $x$-SHC polynomial, namely 
\begin{equation} \label{eqn:unique}
\hat{h} = \hat{h}(x) = \prod_{k\geq1} \prod_{p \in \I_k} p^{a_k} \quad \text{where} \quad a_k = a_k(x) = \left\lfloor \frac{1}{q^{k/x} - 1} \right \rfloor
\end{equation}
Moreover, $\hat{h}$ is the unique highly composite polynomial of its degree.
\item \label{SHCconsec} If $x' < x''$ are two consecutive elements of $S$, then $\hat{h}(x) = \hat{h}(x')$ for all $x' \leq x < x''$. So, there is a one-to-one correspondence between $S$ and the set of superior highly composite polynomials, given by $ x \to \hat{h}(x) $.
\end{enumerate}
\end{theorem}

and that of the semi-superior highly composite polynomials

\begin{theorem} \label{mainHC}
Let $x>0$.
\begin{enumerate}
\item If $x \not \in S$, then there is only one $x$-SSHC polynomial, namely the polynomial $\hat{h}(x)$ defined in equation \eqref{eqn:unique}.
\item \label{SSHCmany} If $x = \frac{s \log q}{\log(1 + 1/r)} \in S $, then there are $2^{\pi(s)}$ $x$-SSHC polynomials of the form
$$ h(x) = \frac{\hat{h}(x)}{P_{i_1} \cdots P_{i_v}} $$
where $\hat{h}(x)$ is as in equation \eqref{eqn:unique}, $0 \leq v \leq \pi(s)$, $P_{i_1}, \cdots, P_{i_v} \in \I_s$ distinct, and $\deg h (x) = \deg \hat{h}(x) - vs$. Moreover, the unique polynomials $h$ given by $v = \pi(s)$ and $v=0$ are two (distinct) consecutive superior highly composite polynomials.
\item If $h(x)$ is $x$-SSHC and $g \in \M_{\deg h(x)}$ is a highly composite polynomial, then $g$ is also $x$-SSHC.
\end{enumerate}
\end{theorem}

\gap

This family of semi-superior highly composite polynomials is not too sparse, so we can use it to construct polynomials at every degree which make the divisor function close to its maximum. In particular, if we let $T(N) := \max\{\tau(f) \ | \ f \in \M_N \}$, then we are able to compute $\log T(N)$ to within an error of at most $\log\frac{4}{3}$:

\begin{theorem} \label{mainMaxDiv}
Let $x  = \frac{s \log q}{\log(1 + 1/r)} \in S $, $\hat{h} = \hat{h}(x)$ and $a_k = a_k(x)$ be defined as in equation \eqref{eqn:unique}, and $h$ be an $x$-SSHC polynomial of degree $\deg \hat{h}(x) - vs$ with $0 \leq v < \pi(s)$. Then, if $N = \deg h - u$ with $0 \leq u \leq s-1$, we have
$$\log T(N) = \begin{cases} \log \tau(h) &\text{ if } u=0 \\ \log \tau(h) - \epsilon(N) &\text{ otherwise} \end{cases}$$
where $$\frac{u}{s}\log\left(1 + \frac{1}{r}\right) \leq \epsilon(N) \leq  \log\left(1+\frac{1}{a_u}\right) $$
Moreover, the size of this range for $\epsilon(N)$ is at most $\log\left(1+\frac{1}{a_u(a_u + 2)}\right) \leq \log\frac{4}{3}$.
\end{theorem}
\begin{rmk}
From the final sentence of part \ref{SSHCmany} of Theorem \ref{mainHC}, we know that the (distinct) superior highly composite polynomial $\hat{h}(x')$ immediately preceding $\hat{h}(x)$ has degree $\deg \hat{h}(x') = \deg \hat{h}(x) - \pi(s)s$. So, the form of $N$ in Theorem \ref{mainMaxDiv} accounts for all integers between the degrees of these two consecutive superior highly composite polynomials. Therefore, for any $N \geq 1$, we can find $x>0$ so as to express $N$ in the form presented in Theorem \ref{mainMaxDiv}.
\end{rmk}

\gap

\section{Superior highly composite polynomials} \label{SHC-Section}

\gap

We begin by showing, contingent on some auxiliary lemmas proven subsequently, that

\begin{prop} \label{SHCvariants}
For each $x>0$, the function $\frac{\tau(f)}{q^{\deg f /x}}$ is maximised over all $f \in \M$ by (at least one) $h = h(x) \in \M$. Moreover, if we write $h = \prod_{p \in \I} p^{a_p} = \prod_{p \in \I} p^{a_p(x)}$, we have that
\begin{enumerate}
\item \label{SHCprop} If $x \not\in S$, then $a_p(x) = \left\lfloor \frac{1}{q^{\deg p/x} - 1} \right \rfloor$ for each $p \in \I$ and so $h$ is unique.
\item \label{SSHCprop} Else, if $x = \frac{s \log q}{\log(1 + 1/r)} \in S$, so that $r = \frac{1}{q^{s/x} - 1}$, then
$$ a_p(x) = \begin{cases} \left\lfloor \frac{1}{q^{\deg p/x} - 1} \right \rfloor &\text{ if } \deg p \neq s \\ r \text{ or } r-1 &\text{ if } \deg p = s\end{cases} $$ 
and so there are $2^{\pi(s)}$ such polynomials $h$.
\end{enumerate}
\end{prop}
\begin{proof}
From \eqref{eqn:polynomial}, we can write
$$ \frac{\tau(f)}{q^{\deg f /x}} = \prod_{p \in \I} \frac{1+a_p}{q^{a_p \deg p / x}} = \exp\left(\sum_{p \in \I} \log(1+a_p) - \frac{a_ p \deg p \log q}{x}\right)$$
so that to maximise $\frac{\tau(f)}{q^{\deg f /x}}$, for each $p \in \I$ we must maximise the quantity $ \phi_{a_p} := \log(1+a_p) - \alpha a_p $ with $ \alpha = \frac{\deg p \log q}{x}$.  \\

If $x \not\in S$, then $\alpha$ cannot be written as $\log(1+\frac{1}{j})$ for any integer $j$, so by Lemma \ref{phi} we have that $\phi_{a_p}$ is maximised if and only if $a_p = \left\lfloor \frac{1}{e^\alpha - 1} \right \rfloor = \left\lfloor \frac{1}{q^{\deg p/x} - 1} \right \rfloor $. \\

Otherwise, if $x\in S$, then by Lemma \ref{unique} there is a unique pair $(s, r)$ such that $x = \frac{s \log q}{\log(1 + 1/r)}$. Therefore, if $ \deg p \neq s $, then $\alpha$ cannot be written as $\log(1+\frac{1}{j})$ for any integer $j$, so by Lemma \ref{phi} we have that $\phi_{a_p}$ is maximised if and only if $a_p = \left\lfloor \frac{1}{e^\alpha - 1} \right \rfloor = \left\lfloor \frac{1}{q^{\deg p/x} - 1} \right \rfloor $. Else, if $ \deg p = s $, then $ \alpha = \log(1+\frac{1}{r}) $ and so by Lemma \ref{phi} we have that $\phi_{a_p}$ is maximised if and only if $a_p = r$ or $r-1$.
\end{proof}
\begin{rmk}
Notice in both cases that $a_p(x)$ is zero for $\deg p > \frac{x \log 2}{\log q}$, so that the factorisation of $h$ is in fact a finite product.
\end{rmk}

\gap

This leads us first to the proof of Theorem \ref{mainSHC}:

\begin{proof}[Proof of Theorem \ref{mainSHC}] For $x>0$, let $\hat{h} = \hat{h}(x) = \prod_{k\geq1} \prod_{p \in \I_k} p^{a_k}$ with $a_k(x) = \left\lfloor \frac{1}{q^{k/x} - 1} \right \rfloor$, and for $x = \frac{s \log q}{\log(1 + 1/r)} \in S$, let $E = E(x)$ be the set of polynomials defined in part \ref{SSHCprop} of Proposition \ref{SHCvariants}.
\begin{enumerate}
\item If $x \not \in S$, then from part \ref{SHCprop} of Proposition \ref{SHCvariants}, we know that for all $f \in \M$, we have
$$ \frac{\tau(\hat{h})}{q^{\deg \hat{h}/x}} > \frac{\tau(f)}{q^{\deg f/x}} $$
and so $\hat{h}$ is the unique $x$-SHC. Moreover, if $\deg f \leq \deg \hat{h}$ then
$$ \tau(\hat{h}) > \tau(f)q^{(\deg \hat{h} - \deg f)/x} \geq \tau(f) $$
and so $\hat{h}$ is the unique highly composite polynomial of its degree. \\
Otherwise, if $x = \frac{s \log q}{\log(1 + 1/r)} \in S$, so that $r = \frac{1}{q^{s/x} - 1}$, we observe that $\hat{h} = \prod_{p \in \I} p^{\hat{a}_p}$ with 
$$ \hat{a}_p = \hat{a}_p(x) = \left\lfloor  \frac{1}{q^{\deg p/x} - 1} \right \rfloor = \begin{cases} \left\lfloor \frac{1}{q^{\deg p/x} - 1} \right \rfloor &\text{ if } \deg p \neq s \\ r  &\text{ if } \deg p = s\end{cases} $$ 
so that $\hat{h}(x) \in E$. Therefore, by part \ref{SSHCprop} of Proposition \ref{SHCvariants}, we have that 
$$ \frac{\tau(\hat{h})}{q^{\deg \hat{h}/x}} \begin{cases} = \frac{\tau(f)}{q^{\deg f/x}} &\text{ if } f \in E \\ > \frac{\tau(f)}{q^{\deg f/x}} &\text{ if } f \not \in E \end{cases}.$$
and that for all $f \in E\setminus\{\hat{h}\}$, we have $\deg f \leq \deg \hat{h} - s  < \deg \hat{h}$. Therefore $\hat{h}$ is the unique $x$-SHC, and moreover, if $\deg f \leq \deg \hat{h}$ then 
$$ \tau(\hat{h}) \begin{cases} = \tau(f)q^{(\deg \hat{h} - \deg f)/x} > \tau(f) &\text{ if } f \in E \\ > \tau(f)q^{(\deg \hat{h} - \deg f)/x} \geq \tau(f) &\text{ if } f \not \in E \end{cases}  $$
and so $\hat{h}$ is the unique highly composite polynomial of its degree.
\item Let $x' < x''$ be two consecutive elements of $S$, and let $\tilde{x} = \min\{x > x' : \hat{h}(x) \neq \hat{h}(x') \}$. Then there is some $\tilde{k}$ such that $a_{\tilde{k}}(\tilde{x}) = m > a_{\tilde{k}}(x')$. Therefore we must have that
$$ x' < \frac{\tilde{k} \log q}{\log\left(1 + \frac{1}{m}\right)} \leq \tilde{x} < \frac{\tilde{k} \log q}{\log\left(1 + \frac{1}{1+m}\right)} $$
and by the minimality of $\tilde{x}$ we must have that $$ \tilde{x} = \frac{\tilde{k} \log q}{\log\left(1 + \frac{1}{m}\right)} \in S.$$ So, by the minimality of $\tilde{x}$  and the definition of $x''$, we conclude that $\tilde{x} = x''$, and therefore that $\hat{h}(x) = \hat{h}(x')$ for all $x' \leq x < x''$. It follows that there is a one-to-one correspondence between $S$ and the set of superior highly composite polynomials, given by $ x \to \hat{h}(x) $.
\end{enumerate}
\end{proof}

and then to the proof of Theorem \ref{mainHC}:

\begin{proof}[Proof of Theorem \ref{mainHC}] For $x>0$, let $\hat{h} = \hat{h}(x) = \prod_{k\geq1} \prod_{p \in \I_k} p^{a_k}$ with $a_k(x) = \left\lfloor \frac{1}{q^{k/x} - 1} \right \rfloor$, and for $x = \frac{s \log q}{\log(1 + 1/r)} \in S$, let $E = E(x)$ be the set of polynomials defined in part \ref{SSHCprop} of Proposition \ref{SHCvariants}.
\begin{enumerate}
\item \label{SHCinHC} If $x \not \in S$, then the result follows from part \ref{SHCprop} of Proposition \ref{SHCvariants}.
\item If $x = \frac{s \log q}{\log(1 + 1/r)} \in S$, then from part \ref{SSHCprop} of Proposition \ref{SHCvariants}, we have that the $x$-SSHC polynomials are precisely the $2^{\pi(s)}$ polynomials in the set $E$, which we can rewrite as
$$ E = \left\{h(x) = \frac{\hat{h}(x)}{P_{i_1} \cdots P_{i_v}} : \ 0 \leq v \leq \pi(s) \text{ and }  P_{i_1}, \cdots, P_{i_v} \in \I_s \text{ distinct}\right\}.$$
When $v = 0$, $h(x) = \hat{h}(x)$ is superior highly composite. When $v = \pi(s)$, $h(x) = \prod_{k\geq1} \prod_{p \in \I_k} p^{\tilde{a}_k}$ where
$$ \tilde{a}_k = \tilde{a}_k(x) = \begin{cases}  a_k(x) &\text{ if } k \neq s \\  r - 1  &\text{ if } k = s\end{cases}. $$ 
Now, let $y = \max\{x' \in S: x' < x\}$ so that, by the definition of $x$ and $y$, we have
$$ \frac{s \log q}{\log\left(1 + \frac{1}{r-1}\right)}  \leq y < x =  \frac{s \log q}{\log\left(1 + \frac{1}{r}\right)}   $$
and so $a_k(y) = r-1$. When $k \neq s$, by Lemma \ref{unique}, we cannot have that $x = \frac{k \log q}{\log(1 + 1/a_k(x))} \in S$, and so
$$ \frac{k \log q}{\log\left(1 + \frac{1}{a_k(x)}\right)} < x <  \frac{k \log q}{\log\left(1 + \frac{1}{1 + a_k(x)}\right)}.$$
This means that, by the definition of $y$, we have for $k \neq s$ that 
$$ \frac{k \log q}{\log\left(1 + \frac{1}{a_k(x)}\right)}  \leq y < x < \frac{k \log q}{\log\left(1 + \frac{1}{1 + a_k(x)}\right)} $$
and so $a_k(y) = a_k(x)$. Therefore, we have that $\hat{h}(y) = h(x)$ and so, by part \ref{SHCconsec} of Theorem \ref{mainSHC}, $h(x)$ is the (distinct) superior highly composite  polynomial immediately preceding $\hat{h}(x)$.
\item Let $h(x)$ be $x$-SSHC and $g \in \M_{\deg h(x)}$ be a highly composite polynomial. If $x \not \in S$, then by part \ref{SHCinHC} of Theorem \ref{mainHC}, $h(x)=\hat{h}(x)$, and by part \ref{SHCunique} of Theorem \ref{mainSHC}, $\hat{h}(x)$ is the unique highly composite polynomial of its degree, so $g = \hat{h}(x) = h(x)$. Else, if $x \in S$, then by part \ref{SSHCprop} of Proposition \ref{SHCvariants}, we have that $h \in E$ and
$$ \frac{\tau(h)}{q^{\deg h/x}} \begin{cases} = \frac{\tau(f)}{q^{\deg f/x}} &\text{ if } f \in E \\ > \frac{\tau(f)}{q^{\deg f/x}} &\text{ if } f \not \in E \end{cases}.$$
Therefore $\tau(g) = \tau(h)$ if, and only if, $g \in E$ which means that $g$ is also $x$-SSHC.
\end{enumerate}
\end{proof}

\gap

Finally we conclude by proving the auxiliary lemmas used in the proof of Proposition \ref{SHCvariants}, namely

\begin{lemma} \label{phi}
Let $ \alpha > 0$ and consider the sequence $(\phi_n)_{n \geq 0} $ defined by $\phi_n = \log(1+n) - \alpha n$. We have that 
\begin{enumerate}
\item If $\log(1+\frac{1}{j+1}) < \alpha < \log(1+\frac{1}{j})$ for some integer $j$, then $\phi_{n}$ is maximised if and only if $n = j = \left\lfloor \frac{1}{e^\alpha - 1} \right\rfloor$.
\item Else, if $ \alpha = \log(1+\frac{1}{j})$ for some integer $j$, then $\phi_{n}$ is maximised if and only if $n = j$ or $j-1$.
\end{enumerate}
\end{lemma}
\begin{proof}
Let $\Delta_n = \phi_n - \phi_{n-1} = \log(1 + \frac{1}{n}) - \alpha$ for $n \geq 1$. Then we have that
\begin{enumerate}
\item If $\log(1+\frac{1}{j+1}) < \alpha < \log(1+\frac{1}{j})$, then $\Delta_n > 0$ when $n \leq j$, and $\Delta_n < 0$ when $n > j$.
\item Else, if $ \alpha = \log(1+\frac{1}{j})$ then $\Delta_n > 0$ when $n < j$, $\Delta_j = 0$ and $\Delta_n < 0$ when $n > j$.
\end{enumerate}
\end{proof}

and

\begin{lemma} \label{unique}
Let $x \in S$. Then there is a unique pair $(s, r)$ such that $x = \frac{s \log q}{\log(1 + 1/r)}$.
\end{lemma}
\begin{proof}
Suppose otherwise, so that $x = \frac{s \log q}{\log(1 + 1/r)} =  \frac{S \log q}{\log(1 + 1/R)}$ with $(r, s)$ and $(R, S)$ distinct. If $s = S$ then $r=R$, so it must be that $s \neq S$ and in particular we may assume without loss of generality that $S>s$ so that $\frac{S}{s} > 1$. \\
Now $\frac{s \log q}{\log(1 + 1/r)} =  \frac{S \log q}{\log(1 + 1/R)}$ implies $\left(1+\frac{1}{R}\right)^s = \left(1+\frac{1}{r}\right)^S$ and therefore $\frac{(R+1)^s}{R^s}  = \frac{(r+1)^S}{r^S}$.
However, $\frac{(R+1)^s}{R^s} $ and $ \frac{(r+1)^S}{r^s}$ are irreducible fractions, so we must have that $R^s = r^S $ and $(R+1)^s = (r+1)^S$. So, $(1+r)^{S/s} = 1 + R = 1 + r^{S/s}$, but $(1+x)^{\alpha} > 1 + x^\alpha$ for $x > 0$ and $\alpha > 1$, which is a contradiction.
\end{proof}

\gap

\section{The maximum value of the divisor function at each degree} \label{maxDiv-Section}

\gap

Since the family of semi-superior highly composite polynomials is not too sparse, we can use it to construct polynomials at every degree which make the divisor function close to its maximum, and thus prove Theorem \ref{mainMaxDiv}:

\begin{proof}[Proof of Theorem \ref{mainMaxDiv}]
If $u=0$, then $N = \deg h$, and by Theorem \ref{mainHC}, $h$ is highly composite, so $T(N) = \tau(h)$. Otherwise, if $1 \leq u \leq s-1$, we have by Remark \ref{increasing} that $a_u(x) \geq a_s(x) = r \geq 1$. So, if we pick $P \in \I_u$ and let $g = \frac{h}{P}$, then $g \in \M$ with $\deg g = \deg h - u = N$ and therefore 
$$ T(N) \geq \tau(g) = \tau(h)\frac{a_u}{1+a_u} = \frac{\tau(h)}{1 + \frac{1}{a_u}} .$$
On the other hand, by Theorem \ref{mainHC}, $h$ is $x$-SSHC, and so for any $f \in \M_N$, we have
$$ \tau(f) \leq \tau(h) q^{(\deg f - \deg h)/x} = \tau(h) q^{-u/x}  $$
and therefore
$$ T(N) \leq   \tau(h) q^{-u/x}   =  \tau(h)  \left(1 + \frac{1}{r}\right)^\frac{-u}{s} .$$
Overall, this gives us that
\begin{equation} \label{eqn:epsilon}
\log q^{u/x} = \frac{u}{s}\log\left(1 + \frac{1}{r}\right) \leq \log \tau(h) - \log T(N) \leq  \log\left(1+\frac{1}{a_u}\right).
\end{equation}
Now, since
$$ a_u = \left\lfloor \frac{1}{q^{u/x} - 1} \right\rfloor \geq \frac{1}{q^{u/x} - 1} - 1$$ 
we have that
$$ q^{u/x} \geq 1 + \frac{1}{1 + a_u} $$
and so the size of the range in equation $\eqref{eqn:epsilon}$ is at most
$$ \log\left(1+\frac{1}{a_u}\right) - \log\left(1+\frac{1}{1 + a_u}\right) = \log\left(1+\frac{1}{a_u(a_u + 2)}\right) \leq \log\frac{4}{3} .$$
\end{proof}

\gap

\appendix \section{Table of highly composite polynomials in $\F_2[t]$} \label{appendix}

\gap

We conclude with a table of highly composite polynomials in $\F_2[t]$, in which SSHC polynomials are additionally marked with $\ast$ and SHC polynomials are additionally marked with $\ast \ast$. We denote the monic irreducible polynomials in $\F_2[t]$ by $P_1(t), P_2(t), \cdots$ in ascending order of the value which they take on $t=2$ (so that, if $\deg P_i > \deg P_j$, then $i > j$), and we write $f \in \M$ in the form $f=P_1^{a_1} P_2^{a_2} \cdots $ in order to shorten the printing. We have listed the explicit values of $P_1(t), \cdots, P_{14}(t)$, which are all of the irreducible polynomials which appear in the factorisations of polynomials in our table of highly composite polynomials, in their own table below, along with the values which they take on $t=2$.

\begin{center}
	\begin{longtable}{|c|c|c|c|}
	\caption{Table of ordered monic irreducible polynomials in $\F_2[t]$} \label{tableIrred} \\
	\hline
	$ i $ & $ P_i(t)  \in \I $ & $\deg P_i $ & $P_i(2)$	\\ \hline
	1 & $t$ & 1 & 2 \\ \hline
	2 & $t + 1$ & 1 & 3 \\ \hline
	3 & $t^2 + t + 1$ & 2 & 7 \\ \hline
	4 & $t^3 + t + 1$ & 3 & 11 \\ \hline
	5 & $t^3 + t^2 + 1$ & 3 & 13 \\ \hline
	6 & $t^4 + t + 1$ & 4 & 19 \\ \hline
	7 & $t^4 + t^3 + 1$ & 4 & 25 \\ \hline
	8 &  $t^4 + t^3 + t^2 + t + 1$ & 4 &31 \\ \hline
	9 &  $t^5 + t^2 + 1$ & 5 & 37 \\ \hline
	10 & $t^5 + t^3 + 1$ & 5 & 41 \\ \hline
	11 & $t^5 + t^3 + t^2 + t + 1$ & 5 & 47 \\ \hline
	12 & $t^5 + t^4 + t^2 + t + 1$ & 5 & 55 \\ \hline
	13 & $t^5 + t^4 + t^3 + t + 1$ & 5 & 59 \\ \hline
 	14 & $t^5 + t^4 + t^3 + t^2 + 1$ & 5 & 61 \\ \hline
	\end{longtable}
\end{center}

The algorithm which we use to compute highly composite polynomials is an adaption of the algorithm used to compute highly composite numbers in \cite{Kedlaya}. Though we take $q=2$, it works in the same way for any $\F_q[t]$, and we give a brief description as follows. \\

We first define the set $\M^{(k)} \subset \M$ of polynomials whose prime factors are in $\{P_1, \cdots, P_k\}$, and we call $f \in \M^{(k)}$ a \emph{$k$-highly composite polynomial} if $\tau(f) = \max\{\tau(g) \ | \ g \in \M^{(k)} \text{ and } \deg g \leq \deg f\}$. 
Then we let $\text{HC}(k, n) \subset \M^{(k)}$ be the set of $k$-highly composite polynomials of degree exactly $n$, and make the following observations:

\begin{itemize}
\item If $f = P_1^{a_1} P_2^{a_2} \cdots P_{k-1}^{a_{k-1}} P_k^{a_k} \in \text{HC}(k, n)$, then $g = P_1^{a_1} P_2^{a_2} \cdots P_{k-1}^{a_{k-1}} \in \text{HC}(k-1, n - a_k \deg P_k)$. Otherwise, if we had some $h \in \text{HC}(k-1, n - a_k \deg P_k)$ with $\tau(h) > \tau(g)$ then we would have $\tau(h P_k^{a_k})  > \tau(f)$ even though $h P_k^{a_k} \in \M^{(k)}$ with $\deg h P_k^{a_k} = n$, which would be a contradiction.
\item If $f \in \text{HC}(k, n)$, and $P_i^{a_i}P_j^{a_j}$ divides $f$  with $\deg P_j > \deg P_i$, then $a_i \geq a_j$. The proof of this is identical to that presented in Remark \ref{increasing}.
\end{itemize}

\gap

The first observation allows us to iteratively compute $\text{HC}(k, n)$, as long as we know $\text{HC}(k-1, m)$ for all $m \leq n$. In particular, for each $j \geq 0$ with $n-j \deg P_k \geq 0$, we pick any (one) $g_j \in \text{HC}(k-1, n - j \deg P_k)$, and determine which values of $j$ maximise $\tau(g_j P_k^j)$. Once we have determined such a set $J = \{j_1, \cdots, j_r\}$ we can conclude that
$$\text{HC}(k, n) = \{f_j P_k^j : j \in J, \ f_j \in \text{HC}(k-1, n - j \deg P_k)\}.$$
It is trivial to observe that the base case $\text{HC}(1, n) = \{P_1^n\}$ for all $n \geq 0$, and we proceed inductively from there. \\

The second observation allows to note that, if $\{P_1, \cdots, P_k\} = \cup_{a \leq b} \I_a$ for some $b \geq 1$, and $\deg P_1 \cdots P_k \geq n$, then the set $\text{HC}(k, n)$ is in fact the set of highly composite polynomials of degree $n$. Thus, once we have $\text{HC}(k, n)$, by taking $k$ sufficiently large, we are able to compute the highly composite polynomials of degree $n$.

\gap

\begin{center}
    \begin{longtable}{| c | c | c | }
    \caption{Table of highly composite polynomials in $\F_2[t]$} \label{tableHCP} \\
    \hline
    $f \in \M$ & $\deg f$ & $\tau(f)$ \\ \hline
   	$ ^{*}P_{1}^{1} $ & 1 & 2 \\ \hline
   	$ ^{*}P_{2}^{1} $ & 1 & 2 \\ \hline
   	$ ^{**}P_{1}^{1} P_{2}^{1} $ & 2 & 4 \\ \hline
   	$ ^{*}P_{1}^{2} P_{2}^{1} $ & 3 & 6 \\ \hline
   	$ ^{*}P_{1}^{1} P_{2}^{2} $ & 3 & 6 \\ \hline
   	$ ^{**}P_{1}^{2} P_{2}^{2} $ & 4 & 9 \\ \hline
   	$ P_{1}^{3} P_{2}^{2} $ & 5 & 12 \\ \hline
   	$ P_{1}^{2} P_{2}^{3} $ & 5 & 12 \\ \hline
   	$ P_{1}^{2} P_{2}^{1} P_{3}^{1} $ & 5 & 12 \\ \hline
   	$ P_{1}^{1} P_{2}^{2} P_{3}^{1} $ & 5 & 12 \\ \hline
   	$ ^{**}P_{1}^{2} P_{2}^{2} P_{3}^{1} $ & 6 & 18 \\ \hline
   	$ ^{*}P_{1}^{3} P_{2}^{2} P_{3}^{1} $ & 7 & 24 \\ \hline
   	$ ^{*}P_{1}^{2} P_{2}^{3} P_{3}^{1} $ & 7 & 24 \\ \hline
   	$ ^{**}P_{1}^{3} P_{2}^{3} P_{3}^{1} $ & 8 & 32 \\ \hline
   	$ P_{1}^{4} P_{2}^{3} P_{3}^{1} $ & 9 & 40 \\ \hline
   	$ P_{1}^{3} P_{2}^{4} P_{3}^{1} $ & 9 & 40 \\ \hline
   	$ P_{1}^{4} P_{2}^{4} P_{3}^{1} $ & 10 & 50 \\ \hline
   	$ ^{*}P_{1}^{3} P_{2}^{3} P_{3}^{1} P_{4}^{1} $ & 11 & 64 \\ \hline
   	$ ^{*}P_{1}^{3} P_{2}^{3} P_{3}^{1} P_{5}^{1} $ & 11 & 64 \\ \hline
   	$ P_{1}^{4} P_{2}^{3} P_{3}^{1} P_{4}^{1} $ & 12 & 80 \\ \hline
   	$ P_{1}^{3} P_{2}^{4} P_{3}^{1} P_{4}^{1} $ & 12 & 80 \\ \hline
   	$ P_{1}^{4} P_{2}^{3} P_{3}^{1} P_{5}^{1} $ & 12 & 80 \\ \hline
   	$ P_{1}^{3} P_{2}^{4} P_{3}^{1} P_{5}^{1} $ & 12 & 80 \\ \hline
   	$ P_{1}^{4} P_{2}^{4} P_{3}^{1} P_{4}^{1} $ & 13 & 100 \\ \hline
   	$ P_{1}^{4} P_{2}^{4} P_{3}^{1} P_{5}^{1} $ & 13 & 100 \\ \hline
   	$ ^{**}P_{1}^{3} P_{2}^{3} P_{3}^{1} P_{4}^{1} P_{5}^{1} $ & 14 & 128 \\ \hline
   	$ ^{*}P_{1}^{4} P_{2}^{3} P_{3}^{1} P_{4}^{1} P_{5}^{1} $ & 15 & 160 \\ \hline
   	$ ^{*}P_{1}^{3} P_{2}^{4} P_{3}^{1} P_{4}^{1} P_{5}^{1} $ & 15 & 160 \\ \hline
   	$ ^{**}P_{1}^{4} P_{2}^{4} P_{3}^{1} P_{4}^{1} P_{5}^{1} $ & 16 & 200 \\ \hline
   	$ P_{1}^{5} P_{2}^{4} P_{3}^{1} P_{4}^{1} P_{5}^{1} $ & 17 & 240 \\ \hline
   	$ P_{1}^{4} P_{2}^{5} P_{3}^{1} P_{4}^{1} P_{5}^{1} $ & 17 & 240 \\ \hline
   	$ P_{1}^{4} P_{2}^{3} P_{3}^{2} P_{4}^{1} P_{5}^{1} $ & 17 & 240 \\ \hline
   	$ P_{1}^{3} P_{2}^{4} P_{3}^{2} P_{4}^{1} P_{5}^{1} $ & 17 & 240 \\ \hline
   	$ ^{**}P_{1}^{4} P_{2}^{4} P_{3}^{2} P_{4}^{1} P_{5}^{1} $ & 18 & 300 \\ \hline
   	$ ^{*}P_{1}^{5} P_{2}^{4} P_{3}^{2} P_{4}^{1} P_{5}^{1} $ & 19 & 360 \\ \hline
   	$ ^{*}P_{1}^{4} P_{2}^{5} P_{3}^{2} P_{4}^{1} P_{5}^{1} $ & 19 & 360 \\ \hline
   	$ ^{**}P_{1}^{5} P_{2}^{5} P_{3}^{2} P_{4}^{1} P_{5}^{1} $ & 20 & 432 \\ \hline
   	$ P_{1}^{6} P_{2}^{5} P_{3}^{2} P_{4}^{1} P_{5}^{1} $ & 21 & 504 \\ \hline
   	$ P_{1}^{5} P_{2}^{6} P_{3}^{2} P_{4}^{1} P_{5}^{1} $ & 21 & 504 \\ \hline
   	$ P_{1}^{4} P_{2}^{4} P_{3}^{2} P_{4}^{1} P_{5}^{1} P_{6}^{1} $ & 22 & 600 \\ \hline
   	$ P_{1}^{4} P_{2}^{4} P_{3}^{2} P_{4}^{1} P_{5}^{1} P_{7}^{1} $ & 22 & 600 \\ \hline
   	$ P_{1}^{4} P_{2}^{4} P_{3}^{2} P_{4}^{1} P_{5}^{1} P_{8}^{1} $ & 22 & 600 \\ \hline
   	$ P_{1}^{5} P_{2}^{4} P_{3}^{2} P_{4}^{1} P_{5}^{1} P_{6}^{1} $ & 23 & 720 \\ \hline
   	$ P_{1}^{4} P_{2}^{5} P_{3}^{2} P_{4}^{1} P_{5}^{1} P_{6}^{1} $ & 23 & 720 \\ \hline
   	$ P_{1}^{5} P_{2}^{4} P_{3}^{2} P_{4}^{1} P_{5}^{1} P_{7}^{1} $ & 23 & 720 \\ \hline
   	$ P_{1}^{4} P_{2}^{5} P_{3}^{2} P_{4}^{1} P_{5}^{1} P_{7}^{1} $ & 23 & 720 \\ \hline
   	$ P_{1}^{5} P_{2}^{4} P_{3}^{2} P_{4}^{1} P_{5}^{1} P_{8}^{1} $ & 23 & 720 \\ \hline
   	$ P_{1}^{4} P_{2}^{5} P_{3}^{2} P_{4}^{1} P_{5}^{1} P_{8}^{1} $ & 23 & 720 \\ \hline
   	$ ^{*}P_{1}^{5} P_{2}^{5} P_{3}^{2} P_{4}^{1} P_{5}^{1} P_{6}^{1} $ & 24 & 864 \\ \hline
   	$ ^{*}P_{1}^{5} P_{2}^{5} P_{3}^{2} P_{4}^{1} P_{5}^{1} P_{7}^{1} $ & 24 & 864 \\ \hline
   	$ ^{*}P_{1}^{5} P_{2}^{5} P_{3}^{2} P_{4}^{1} P_{5}^{1} P_{8}^{1} $ & 24 & 864 \\ \hline
   	$ P_{1}^{6} P_{2}^{5} P_{3}^{2} P_{4}^{1} P_{5}^{1} P_{6}^{1} $ & 25 & 1008 \\ \hline
   	$ P_{1}^{5} P_{2}^{6} P_{3}^{2} P_{4}^{1} P_{5}^{1} P_{6}^{1} $ & 25 & 1008 \\ \hline
   	$ P_{1}^{6} P_{2}^{5} P_{3}^{2} P_{4}^{1} P_{5}^{1} P_{7}^{1} $ & 25 & 1008 \\ \hline
   	$ P_{1}^{5} P_{2}^{6} P_{3}^{2} P_{4}^{1} P_{5}^{1} P_{7}^{1} $ & 25 & 1008 \\ \hline
   	$ P_{1}^{6} P_{2}^{5} P_{3}^{2} P_{4}^{1} P_{5}^{1} P_{8}^{1} $ & 25 & 1008 \\ \hline
   	$ P_{1}^{5} P_{2}^{6} P_{3}^{2} P_{4}^{1} P_{5}^{1} P_{8}^{1} $ & 25 & 1008 \\ \hline
   	$ P_{1}^{4} P_{2}^{4} P_{3}^{2} P_{4}^{1} P_{5}^{1} P_{6}^{1} P_{7}^{1} $ & 26 & 1200 \\ \hline
   	$ P_{1}^{4} P_{2}^{4} P_{3}^{2} P_{4}^{1} P_{5}^{1} P_{6}^{1} P_{8}^{1} $ & 26 & 1200 \\ \hline
   	$ P_{1}^{4} P_{2}^{4} P_{3}^{2} P_{4}^{1} P_{5}^{1} P_{7}^{1} P_{8}^{1} $ & 26 & 1200 \\ \hline
   	$ P_{1}^{5} P_{2}^{4} P_{3}^{2} P_{4}^{1} P_{5}^{1} P_{6}^{1} P_{7}^{1} $ & 27 & 1440 \\ \hline
   	$ P_{1}^{4} P_{2}^{5} P_{3}^{2} P_{4}^{1} P_{5}^{1} P_{6}^{1} P_{7}^{1} $ & 27 & 1440 \\ \hline
   	$ P_{1}^{5} P_{2}^{4} P_{3}^{2} P_{4}^{1} P_{5}^{1} P_{6}^{1} P_{8}^{1} $ & 27 & 1440 \\ \hline
   	$ P_{1}^{4} P_{2}^{5} P_{3}^{2} P_{4}^{1} P_{5}^{1} P_{6}^{1} P_{8}^{1} $ & 27 & 1440 \\ \hline
   	$ P_{1}^{5} P_{2}^{4} P_{3}^{2} P_{4}^{1} P_{5}^{1} P_{7}^{1} P_{8}^{1} $ & 27 & 1440 \\ \hline
   	$ P_{1}^{4} P_{2}^{5} P_{3}^{2} P_{4}^{1} P_{5}^{1} P_{7}^{1} P_{8}^{1} $ & 27 & 1440 \\ \hline
   	$ ^{*}P_{1}^{5} P_{2}^{5} P_{3}^{2} P_{4}^{1} P_{5}^{1} P_{6}^{1} P_{7}^{1} $ & 28 & 1728 \\ \hline
   	$ ^{*}P_{1}^{5} P_{2}^{5} P_{3}^{2} P_{4}^{1} P_{5}^{1} P_{6}^{1} P_{8}^{1} $ & 28 & 1728 \\ \hline
   	$ ^{*}P_{1}^{5} P_{2}^{5} P_{3}^{2} P_{4}^{1} P_{5}^{1} P_{7}^{1} P_{8}^{1} $ & 28 & 1728 \\ \hline
   	$ P_{1}^{6} P_{2}^{5} P_{3}^{2} P_{4}^{1} P_{5}^{1} P_{6}^{1} P_{7}^{1} $ & 29 & 2016 \\ \hline
   	$ P_{1}^{5} P_{2}^{6} P_{3}^{2} P_{4}^{1} P_{5}^{1} P_{6}^{1} P_{7}^{1} $ & 29 & 2016 \\ \hline
   	$ P_{1}^{6} P_{2}^{5} P_{3}^{2} P_{4}^{1} P_{5}^{1} P_{6}^{1} P_{8}^{1} $ & 29 & 2016 \\ \hline
   	$ P_{1}^{5} P_{2}^{6} P_{3}^{2} P_{4}^{1} P_{5}^{1} P_{6}^{1} P_{8}^{1} $ & 29 & 2016 \\ \hline
   	$ P_{1}^{6} P_{2}^{5} P_{3}^{2} P_{4}^{1} P_{5}^{1} P_{7}^{1} P_{8}^{1} $ & 29 & 2016 \\ \hline
   	$ P_{1}^{5} P_{2}^{6} P_{3}^{2} P_{4}^{1} P_{5}^{1} P_{7}^{1} P_{8}^{1} $ & 29 & 2016 \\ \hline
   	$ P_{1}^{4} P_{2}^{4} P_{3}^{2} P_{4}^{1} P_{5}^{1} P_{6}^{1} P_{7}^{1} P_{8}^{1} $ & 30 & 2400 \\ \hline
   	$ P_{1}^{5} P_{2}^{4} P_{3}^{2} P_{4}^{1} P_{5}^{1} P_{6}^{1} P_{7}^{1} P_{8}^{1} $ & 31 & 2880 \\ \hline
   	$ P_{1}^{4} P_{2}^{5} P_{3}^{2} P_{4}^{1} P_{5}^{1} P_{6}^{1} P_{7}^{1} P_{8}^{1} $ & 31 & 2880 \\ \hline
   	$ ^{**}P_{1}^{5} P_{2}^{5} P_{3}^{2} P_{4}^{1} P_{5}^{1} P_{6}^{1} P_{7}^{1} P_{8}^{1} $ & 32 & 3456 \\ \hline
   	$ ^{*}P_{1}^{6} P_{2}^{5} P_{3}^{2} P_{4}^{1} P_{5}^{1} P_{6}^{1} P_{7}^{1} P_{8}^{1} $ & 33 & 4032 \\ \hline
   	$ ^{*}P_{1}^{5} P_{2}^{6} P_{3}^{2} P_{4}^{1} P_{5}^{1} P_{6}^{1} P_{7}^{1} P_{8}^{1} $ & 33 & 4032 \\ \hline
   	$ ^{**}P_{1}^{6} P_{2}^{6} P_{3}^{2} P_{4}^{1} P_{5}^{1} P_{6}^{1} P_{7}^{1} P_{8}^{1} $ & 34 & 4704 \\ \hline
   	$ P_{1}^{7} P_{2}^{6} P_{3}^{2} P_{4}^{1} P_{5}^{1} P_{6}^{1} P_{7}^{1} P_{8}^{1} $ & 35 & 5376 \\ \hline
   	$ P_{1}^{6} P_{2}^{7} P_{3}^{2} P_{4}^{1} P_{5}^{1} P_{6}^{1} P_{7}^{1} P_{8}^{1} $ & 35 & 5376 \\ \hline
   	$ P_{1}^{6} P_{2}^{5} P_{3}^{3} P_{4}^{1} P_{5}^{1} P_{6}^{1} P_{7}^{1} P_{8}^{1} $ & 35 & 5376 \\ \hline
   	$ P_{1}^{5} P_{2}^{6} P_{3}^{3} P_{4}^{1} P_{5}^{1} P_{6}^{1} P_{7}^{1} P_{8}^{1} $ & 35 & 5376 \\ \hline
   	$ ^{**}P_{1}^{6} P_{2}^{6} P_{3}^{3} P_{4}^{1} P_{5}^{1} P_{6}^{1} P_{7}^{1} P_{8}^{1} $ & 36 & 6272 \\ \hline
   	$ P_{1}^{7} P_{2}^{6} P_{3}^{3} P_{4}^{1} P_{5}^{1} P_{6}^{1} P_{7}^{1} P_{8}^{1} $ & 37 & 7168 \\ \hline
   	$ P_{1}^{6} P_{2}^{7} P_{3}^{3} P_{4}^{1} P_{5}^{1} P_{6}^{1} P_{7}^{1} P_{8}^{1} $ & 37 & 7168 \\ \hline
   	$ P_{1}^{7} P_{2}^{7} P_{3}^{3} P_{4}^{1} P_{5}^{1} P_{6}^{1} P_{7}^{1} P_{8}^{1} $ & 38 & 8192 \\ \hline
   	$ P_{1}^{6} P_{2}^{6} P_{3}^{3} P_{4}^{2} P_{5}^{1} P_{6}^{1} P_{7}^{1} P_{8}^{1} $ & 39 & 9408 \\ \hline
   	$ P_{1}^{6} P_{2}^{6} P_{3}^{3} P_{4}^{1} P_{5}^{2} P_{6}^{1} P_{7}^{1} P_{8}^{1} $ & 39 & 9408 \\ \hline
   	$ P_{1}^{6} P_{2}^{6} P_{3}^{2} P_{4}^{1} P_{5}^{1} P_{6}^{1} P_{7}^{1} P_{8}^{1} P_{9}^{1} $ & 39 & 9408 \\ \hline
   	$ P_{1}^{6} P_{2}^{6} P_{3}^{2} P_{4}^{1} P_{5}^{1} P_{6}^{1} P_{7}^{1} P_{8}^{1} P_{10}^{1} $ & 39 & 9408 \\ \hline
   	$ P_{1}^{6} P_{2}^{6} P_{3}^{2} P_{4}^{1} P_{5}^{1} P_{6}^{1} P_{7}^{1} P_{8}^{1} P_{11}^{1} $ & 39 & 9408 \\ \hline
   	$ P_{1}^{6} P_{2}^{6} P_{3}^{2} P_{4}^{1} P_{5}^{1} P_{6}^{1} P_{7}^{1} P_{8}^{1} P_{12}^{1} $ & 39 & 9408 \\ \hline
   	$ P_{1}^{6} P_{2}^{6} P_{3}^{2} P_{4}^{1} P_{5}^{1} P_{6}^{1} P_{7}^{1} P_{8}^{1} P_{13}^{1} $ & 39 & 9408 \\ \hline
   	$ P_{1}^{6} P_{2}^{6} P_{3}^{2} P_{4}^{1} P_{5}^{1} P_{6}^{1} P_{7}^{1} P_{8}^{1} P_{14}^{1} $ & 39 & 9408 \\ \hline
    \end{longtable}
\end{center}

\begin{rmk}
In $\F_2[t]$, there are certain degrees at which there is a unique highly composite polynomial of that degree, but where that polynomial is neither SHC nor SSHC (see degrees 10, 30 and 38 in the table of highly composite polynomials, for example). We leave for further investigation the question of whether there are infinitely many degrees with this property.
\end{rmk}

\gap

\section*{Acknowledgements}

The author would like to thank Andrew Granville for his encouragement and thoughtful advice, Sam Porritt for useful discussions, contributions and references, and the anonymous referee for suggestions which led to an extensive revision of the entire article. The research leading to these results has received funding from the European Research Council under the European Union's Seventh Framework Programme (FP7/2007-2013), ERC grant agreement n$^{\text{o}}$ 670239.

\gap

\end{document}